\documentclass[12pt]{amsart}

\usepackage{amsthm,amsfonts,amsmath,amssymb,latexsym,epsfig,mathrsfs,yfonts,marvosym}
\usepackage[usenames]{color}
\usepackage{graphicx}
\usepackage[all]{xy}
\usepackage{epsfig}
\usepackage{epic}
\usepackage{eepic}
\usepackage{hyperref}
\usepackage{dsfont}\let\mathbb\mathds
\usepackage[english]{babel}

\DeclareMathAlphabet\oldmathcal{OMS}        {cmsy}{b}{n}
\SetMathAlphabet    \oldmathcal{normal}{OMS}{cmsy}{m}{n}
\DeclareMathAlphabet\oldmathbcal{OMS}       {cmsy}{b}{n}
 
\usepackage{eucal}

\usepackage[active]{srcltx}

\newtheorem{theorem}{Theorem}[section]
\newtheorem{lemma}[theorem]{Lemma}
\newtheorem{proposition}[theorem]{Proposition}
\newtheorem{corollary}[theorem]{Corollary}

\newtheorem{def/prop}[theorem]{Definition/Proposition}
\newtheorem{algorithm}[theorem]{Algorithm}

\newenvironment{example}{\medskip \refstepcounter{theorem}
\noindent  {\bf Example \thetheorem}.\rm}{\,}
\newenvironment{remark}{\medskip \refstepcounter{theorem}
\newcommand     {\comment}[1]   {}
\newcommand{\mute}[2] {}
\newcommand     {\printname}[1] {}

\noindent  {\bf Remark \thetheorem}.\rm}{\,}
\def\<{\langle}

\def\>{\rangle}

\def\BOne{{\mathchoice {\rm 1\mskip-4mu l} {\rm 1\mskip-4mu l}
                          {\rm 1\mskip-4.5mu l} {\rm 1\mskip-5mu l}}}

\def\tr{{\rm tr}~}

\def\fract#1#2{\raise4pt\hbox{$ #1 \atop #2 $}}
\def\decdnar#1{\phantom{\hbox{$\scriptstyle{#1}$}}
\left\downarrow\vbox{\vskip15pt\hbox{$\scriptstyle{#1}$}}\right.}

\def\bbc{{\mathbb C}}

\def\bbp{{\mathbb P}}

\def\bbr{{\mathbb R}}

\def\bbz{{\mathbb Z}}

\def\gra{\alpha}

\def\grg{\gamma}

\def\grk{\kappa}
\def\grl{\lambda}

\def\gro{\omega}

\def\grs{\sigma}

\def\grD{\Delta}

\def\grS{\Sigma}

\def\bfl{{\bf l}}

\def\bfv{{\bf v}}
\def\bfw{{\bf w}}

\def\cala{{\mathcal A}}

\def\calc{{\mathcal C}}

\def\cald{{\mathcal D}}

\def\calf{{\mathcal F}}

\def\cali{{\mathcal I}}

\def\calr{{\mathcal R}}
\def\cals{{\oldmathcal S}}

\def\calw{{\mathcal W}}

\def\la#1{\hbox to #1pc{\leftarrowfill}}
\def\ra#1{\hbox to #1pc{\rightarrowfill}}

\def\gn{{\mathfrak n}}
\def\go{{\mathfrak o}}

\def\gr{{\mathfrak r}}

\def\gt{{\mathfrak t}}

\def\gz{{\mathfrak z}}

\def\gB{{\mathfrak B}}
\def\gC{{\mathfrak C}}

\def\gQ{{\mathfrak Q}}
\def\gR{{\mathfrak R}}

\def\hook{\mathbin{\hbox to 6pt{%
                 \vrule height0.4pt width5pt depth0pt
                 \kern-.4pt
                 \vrule height6pt width0.4pt depth0pt\hss}}}

\def\12{\xi_{k_1,k_2}}
\def\m5{M^5_{k_1,k_2}}

\begin{document}

\title{The Sasaki Join and Admissible K\"ahler Constructions}

\author{Charles P. Boyer and Christina W. T{\o}nnesen-Friedman}\thanks{Both authors were partially supported by grants from the Simons Foundation, CPB by (\#245002) and CWT-F by (\#208799)}
\address{Charles P. Boyer, Department of Mathematics and Statistics,
University of New Mexico, Albuquerque, NM 87131.}
\email{cboyer@math.unm.edu} 
\address{Christina W. T{\o}nnesen-Friedman, Department of Mathematics, Union
College, Schenectady, New York 12308, USA } \email{tonnesec@union.edu}



\maketitle

\markboth{Sasaki Join, Admissible K\"ahler}{Charles P. Boyer and Christina W. T{\o}nnesen-Friedman}


\begin{abstract}
We give a survey of our recent work \cite{BoTo11,BoTo12b,BoTo13,BoTo13b,BoTo14a,BoTo14b} describing a method which combines the Sasaki join construction of \cite{BGO06} with the admissible K\"ahler construction of \cite{ApCaGa06,ACGT04,ACGT08,ACGT08c} to obtain new extremal and new constant scalar curvature Sasaki metrics, including Sasaki-Einstein metrics. The constant scalar curvature Sasaki metrics also provide explicit solutions to the CR Yamabe problem. In this regard we give examples of the lack of uniqueness when the Yamabe invariant $\grl(M)$ is positive.

\end{abstract}

\centerline{Dedicated to our good friend and colleague}
\centerline{Paul Gauduchon on the occasion of his 70th birthday}


\section{Introduction}

It is a distinct honor and privilege  to dedicate this survey of our recent work to our good friend and colleague, Paul Gauduchon. Paul's influence on our work is absolutely apparent and cannot be overestimated. But his influence actually goes much deeper than our recent work, as he has had a very important influence on both of our careers.

In the mid nineteen eighties when the first author was making the transition from mathematical physics to pure mathematics, he was directed to the work of Paul Gauduchon by E. Calabi. Two of Paul's papers, \cite{Gau77a,Gau77b}, subsequently played a seminal role in the first author's work on the geometry of 4-manifolds \cite{Boy86,Boy88,Boy88b}. The first author met Paul Gauduchon in 1990 at the Summer AMS Institute on Differential Geometry in UCLA. A few years later we spent time together at the Erwin Schrodinger Institute in Vienna, and became good friends.

The second author first met Paul Gauduchon during the spring of 2001 at a conference in Denmark. It was then that she first heard about the exciting program of Hamiltonian $2$-forms that put existing K\"ahler metric constructions in a whole new light. Paul generously invited her to spend a week at \'Ecole Polytechique during the summer of 2001 and it was here that her collaboration and friendship with him, Vestislav Apostlov, and David Calderbank had its beginning.  It is impossible to overstate just how much she appreciates having been a part of this program of research and she is acutely aware of the fact that many of her subsequent opportunities in research have been due to this collaboration. It is typical for Paul to reach out to young researchers and help them on their way. He is without a doubt one of the friendliest, most inspiring, and most generous leading differential geometers of our time.

\section{Preliminaries on K\"ahlerian and Sasakian Geometry}

K\"ahlerian and Sasakian geometry are sister geometries; one (K\"ahlerian) lives in even dimensions and the other (Sasakian) lives in odd dimensions. Furthermore, their relation to other important geometries is similar. K\"ahler geometry is a substructure of both complex and symplectic geometry; whereas, Sasakian geometry is a substructure of CR and contact geometry. And both are examples of Riemannian geometry.

A {\it K\"ahler manifold} is a symplectic manifold $(N,\gro)$ which is also a complex manifold such that the complex structure tensor $J$ and symplectic form $\gro$ satisfy the conditions that $\gro(X,JY)$ is positive definite and $\gro(JX,JY)=\gro(X,Y)$ for any vector fields $X,Y$. We also can define a {\it K\"ahler orbifold} for which we refer to Chapter 4 of \cite{BG05} for the precise definition. The symmetric form $\gro(X,JY)$ is a Riemannian metric on $N$ called the {\it K\"ahler metric}.

A {\it Sasaki manifold} is a contact manifold $(M,\eta)$ with contact bundle $\cald=\ker\eta$ which is also a strictly pseudoconvex CR manifold $(\cald,J)$ satisfying $d\eta(JX,JY)=d\eta(X,Y)$ for all sections of $\cald$ and whose Reeb vector field $\xi$ is an infinitesimal CR automorphism. So the Lie algebra of infinitesimal automorphisms of a Sasakian structure is at least one dimensional. If we extend the transverse complex structure $J$ to an endomorphism of the tangent bundle $TM$ by setting $\Phi=J$ on $\cald$ and $\Phi\xi=0$, the symmetric form defined for all vector fields $X,Y$ by $g(X,Y)=d\eta(X,\Phi Y)+\eta(X)\eta(Y)$ is a Riemannian metric on $M$ called the {\it Sasaki metric} for which $\xi$ is Killing vector field. Note that $d\eta\circ (\BOne\times \Phi)$ is a sub-Riemannian K\"ahler metric, called the transverse metric and denoted by $g^T$. It is also called a {\it pseudohermitian metric} in the CR literature. 

When $M$ is a Sasaki manifold and all the Reeb orbits are closed, the Reeb field $\xi$ generates a locally free circle action whose quotient is a K\"ahler orbifold, actually a projective algebraic orbifold. In this case $\xi$ (or the Sasakian structure) is called {\it quasi-regular}. When the action is free $\xi$ is {\it regular} and the quotient is a K\"ahler manifold. This construction, known as the Boothby-Wang construction \cite{BoWa,BG00a}, can be inverted as follows. Given a K\"ahler orbifold with integral cohomology class $[\gro]\in H^2_{orb}(N,\bbz)$, we choose a connection 1-form $\eta$ on the total space $M$ of the $S^1$ orbibundle over $N$. Then $(M,\eta)$ is a Sasaki manifold. Notice that a scale transformation or homothety of the transverse K\"ahler metric, $g^T\mapsto ag^T$ for $a\in\bbr^+$ induces the {\it transverse homothety} $g\mapsto ag+(a^2-a)\eta\otimes\eta$ of the Sasaki metric, giving a one parameter family of new Sasaki metrics.

Even in the {\it irregular} case when there are non-closed Reeb orbits and no reasonable global quotient, the local quotients of the local submersions are K\"ahler. Hence, a Sasaki manifold has a transverse (to the flow of Reeb field) K\"ahlerian structure.

\subsection{Extremal Metrics}
The notion of extremal K\"ahler metrics was introduced as a variational problem by Calabi in \cite{Cal56} and studied in greater depth in \cite{Cal82}. The most effective functional is probably the $L^2$-norm of scalar curvature, viz.
\begin{equation}\label{Calfun}
E(\omega) = \int_M s^2 d\mu,
\end{equation}
where $s$ is the scalar curvature and $d\mu$ is the volume form of the
K\"ahler metric corresponding to the K\"ahler form $\omega$. The variation is taken over the set of all K\"ahler metrics within a fixed K\"ahler class $[\gro]$. Calabi showed that the critical points of the functional $E$ are precisely the K\"ahler metrics such that the gradient vector field $J{\rm grad}~s$ is holomorphic. A recent and detailed account is given in the forthcoming book \cite{Gau09b}.

On the Sasakian level extremal metrics were developed in \cite{BGS06}. The procedure is quite analogous again using the $L^2$-norm of scalar curvature $s_g$ of the Sasaki metric $g$, viz.
\begin{equation}\label{sasfun}
E(g) = \int_M s_g^2 d\mu_g,
\end{equation}
where now the variation is taken over all Sasaki metrics associated to the basic cohomology class $[d\eta_B]\in H^2_B(\calf_\xi)$. Again the critical points are precisely those Sasaki metrics such that the gradient vector field $J{\rm grad}_gs_g$ is transversely holomorphic. Since the scalar curvature $s_g$ is related to the transverse scalar curvature $s^T_g$ of the transverse K\"ahler metric by $s_g=s_g^T-2n$, a Sasaki metric is extremal (CSC) if and only if its transverse K\"ahler metric is extremal (CSC).

\subsection{The Sasaki Cone and Bouquets}
Like the space of K\"ahler metrics belonging to a fixed cohomology class, the space of Sasaki metrics belonging to a fixed isotopy class of contact structure is infinite dimensional \cite{BG05}. However, the space of Sasakian structures belonging to a fixed strictly pseudoconvex CR structure $(\cald,J)$ has finite dimensions. It is called the {\it unreduced Sasaki cone} and is defined as follows. Fix a maximal torus $T$ in the CR automorphism group $\gC\gr(\cald,J)$ of a Sasakian manifold $(M,\eta)$  and let $\gt(\cald,J)$ denote the Lie algebra of $T$. Then the unreduced Sasaki cone is defined by 
\begin{equation}\label{unsascone}
\gt^+(\cald,J)=\{\xi'\in\gt(\cald,J)~|~\eta(\xi')>0\}.
\end{equation}
Then the {\it reduced Sasaki cone} is defined by $\grk(\cald,J)=\gt^+(\cald,J)/\calw(\cald,J)$ where $\calw(\cald,J)$ is the Weyl group of $\gC\gr(\cald,J)$. The reduced Sasaki cone $\grk(\cald,J)$ can be thought of as the moduli space of Sasakian structures whose underlying CR structure is $(\cald,J)$. We shall often suppress the CR notation $(\cald,J)$ when it is understood from the context.

Now for each strictly pseudoconvex CR structure $(\cald,J)$ there is a unique conjugacy class of maximal tori in $\gC\gr(\cald,J)$. This in turn defines a conjugacy class $\calc_T(\cald)$ of tori in the contactomorphism group $\gC\go\gn(\cald)$ which may or may not be maximal. We thus recall \cite{Boy10a} the map $\gQ$ that associates to any transverse almost complex structure $J$ that is compatible with the contact structure $\cald$, a conjugacy class of tori in $\gC\go\gn(\cald)$, namely the unique conjugacy class of maximal tori in $\gC\gR(\cald,J)\subset \gC\go\gn(\cald)$. Then two compatible transverse almost complex structures $J,J'$ are {\it T-equivalent} if $\gQ(J)=\gQ(J')$. A {\it Sasaki bouquet} is defined by
\begin{equation}\label{sasbouq}
\gB_{|\cala|}(\cald)=\bigcup_{\gra\in\cala}\grk(\cald,J_\gra)
\end{equation}
where the union is taken over one representative of each $T$-equivalence class in a preassigned subset $\cala$ of $T$-equivalence classes of transverse (almost) complex structures. Here $|\cala|$ denotes the cardenality of $\cala$.

\section{The Admissible Construction}\label{admissec}

The admissible construction of K\"ahler metrics, which we have utilized to obtain our results, may be viewed as a special case of the construction of K\"ahler metrics admitting a so-called {\it Hamiltonian 2-form}. This term was introduced by V. Apostolov, D. Calderbank, and P. Gauduchon in  \cite{ApCaGa06}. We will here give a very brief overview/history of this topic.

Let $(S,J,\gro,g)$ be a K\"ahler manifold of real dimension $2m$. Recall \cite{ApCaGa06} that on $(S,J,\gro,g)$ a {\it Hamiltonian 2-form} is a $J$-invariant 2-form $\phi$ that satisfies the differential equation
\begin{equation}\label{ham2form}
2\nabla_X\phi = d\tr\phi\wedge (JX)^\flat-d^c\tr\phi\wedge X^\flat
\end{equation}
for any vector field $X$. Here $X^\flat$ indicates the 1-form dual to $X$, and $\tr\phi$ is the trace with respect to the K\"ahler form $\gro$, i.e. $\tr\phi=g(\phi,\gro)$ where $g$ is the K\"ahler metric.
Note that if $\phi$ is a Hamiltonian $2$-form, then so is $\phi_{a,b} = a\phi + b\omega$ for any constants $a,b \in \bbr$.

Hamiltonian $2$-forms occur naturally for {\it Weakly-Bochner-flat} (WBF) {\it K\"ahler metrics} as follows. A WBF K\"ahler metric is defined to be a K\"ahler metric whose Bochner tensor
(which is part of the curvature tensor) is co-closed. By using the differential
Bianchi identity one can see that this condition is equivalent (for $m\geq 2$) to the condition that
$\rho + \frac{s}{2 (m + 1)} \, \omega$ is a Hamiltonian $2$-form, where
$\rho$ is the Ricci form and $s$ is the scalar curvature.  WBF K\"ahler metrics are in particular extremal K\"ahler metrics
and they are generalization of  Bochner-flat K\"ahler metrics, studied by Bryant~\cite{Bry01}, and products
of K\"ahler--Einstein metrics. Note that for $m=2$, WBF K\"ahler metrics and Bochner-flat K\"ahler metrics are called 
{\it Weakly Selfdual K\"ahler metrics} \cite{ApCaGa03} and {\it Selfdual K\"ahler metrics} respectively.

A Hamiltonian $2$-form $\phi$ induces an isometric Hamiltonian $l$-torus action on $S$ for some $0\leq l \leq m$. 
This follows from \cite{ApCaGa06} where the K\"ahler form
 $\omega$ is used to identify $\phi$ with a Hermitian endomorphism.  They then consider  the elementary symmetric functions $\grs_1,\ldots,\grs_n$ of its $n$ eigenvalues. The Hamiltonian vector fields $K_i=J{\rm grad}\grs_i$ are Killing with respect to the K\"ahler 
 metric $g$. Moreover the Poisson brackets $\{\sigma_i,\sigma_j\}$ all vanish and so, in particular the
 vector fields $K_1,...,K_m$ commute. In the case where $K_1,...,K_m$ are independent, we have that $(S,J,\gro,g)$ is toric. In fact, it is a very special kind of toric, namely {\em orthotoric} \cite{ApCaGa06}. In general, it is proved in  \cite{ApCaGa06} that there exists a number
 $0 \leq l \leq m$ such that  the span of $K_1,...,K_m$ is everywhere at most $l-dimensional$ and, on an open dense set $S^0$, $K_1,...,K_l$ are linearly independent.
 Now $l$ is called the {\it order} of $\phi$.  

Using the Pedersen-Poon ansatz  for K\"ahler metrics with a local isometric hamitonian $l$-torus action \cite{PePo91},  a local classification of K\"ahler metrics admitting a Hamiltonian $2$-form was then obtained in \cite{ApCaGa06}. This local classification then gave a specialized ansatz that e.g would simplify the PDE system equivalent to the  extremal K\"ahler metric condition to a much more amenable ODE system. Previously this specialized ansatz had been successfully assumed by many authors (\cite{Cal82}, \cite{KoSa88}, \cite{PePo91}, \cite{LeB91b},\cite{Sim92}, \cite{Hwa94}, \cite{Gua95}, \cite{To-Fr98}, and \cite{HwaSi02}, to name a few - with apologies to anyone we left out)  in order to produce examples of
 K\"ahler  metrics with special geometric properties. However, to see that this assumption was naturally given in the form of the existence of a  Hamiltonian $2$-form, was indeed a spectacular observation in \cite{ApCaGa06}. In particular, as a consequence it became clear that this specialization was not only useful, but also necessary in the case of WBF metric constructions.
 
The natural continuation of \cite{ApCaGa06} was to move from a local to a global classification of K\"ahler metrics admitting a Hamiltonian $2$-form. This was done in \cite{ACGT04}, and hence, a blueprint for the construction of not only local but global (compact) examples of K\"ahler metrics with various geometric properties was established.

\subsection{K\"ahler Admissibility and Ruled Manifolds}\label{adsec}
We will focus on a special case in the $l=1$ case, which according to \cite{ACGT08} would be called something like
 {\it admissible with no blow-downs and only one piece in the base}. For simplicity we will abuse this terminology a bit and simply refer to our case as {\it admissible}. 
 
Assume that $n\in \bbz\setminus\{0\}$ and $(N,\gro_N,g_N)$  is a compact K\"ahler structure with CSC K\"ahler metric $g_N$.
Then 
$(\omega_{N_n},g_{N_n}): =(2n\pi \omega_N, 2n\pi g_N)$ satisfies that 
$( g_{N_n}, 
\omega_{N_n})$ or $(- g_{N_n}, 
-\omega_{N_n})$
is a K\"ahler structure (depending on the sign of $n$). 
In either case, we let $(\pm g_{N_n}, \pm \omega_{N_n})$ refer to the K\"ahler structure.
We denote the real dimension of $N$ by
$2 d_{N}$ and write the scalar curvature of $\pm g_{N_n}$ as $\pm 2 d_{N_n} s_{N_n}$.
[So, if e.g. $-g_{N_n}$ is a K\"ahler structure with positive scalar curvature, $s_{N_n}$ would be negative.]

Now for a holomorphic line bundle $L_n \rightarrow N$ such that
$c_{1}(L_n)= [\omega_{N_n}/2\pi]$.
the total space of the projectivization
$S_n=\bbp(\BOne\oplus L_n)$ is called {\it admissible}.
On these manifolds, K\"ahler metrics admitting a Hamiltonian $2$-form of order one can be constructed in such a way that the natural fiberwise $S^1$-action is induced by the Hamiltonian vector field arising from the Hamiltonian $2$-form \cite{ACGT08}. These metrics, as well as the K\"ahler classes of their K\"ahler forms, are also called {\it admissible}. 

In the case where $N$ is a Riemann surface, the entire K\"ahler cone consist of admissible K\"ahler classes but in general the set of admissible K\"ahler classes constitutes a subcone of the K\"ahler cone. Up to scale, an admissible K\"ahler class $\Omega_{\mathbf r}$ is determined by a real number $r$ of the same sign as
$g_{N_n}$ and satisfying $0 < |r| < 1$. More specifically
\begin{equation}\label{admKahclass}
 \Omega_{\mathbf r}  = [\omega_{N_n}]/r + 2 \pi PD[D_1+D_2],
 \end{equation}
where $PD$ denotes the Poincar\'e dual and the divisors $D_1$ and $D_2$ are given by the zero section $\BOne\oplus 0$ and infinity section $0\oplus L_n$, respectively. Consider the circle action
on $S_n$ induced by the natural circle action on $L_n$. It extends to a holomorphic
$\mathbb{C}^*$ action. The open and dense set ${S_n}_0\subset S_n$ of stable points with respect to the
latter action has the structure of a principal circle bundle over the stable quotient.
The Hermitian norm on the fibers induces, via a Legendre transform, a function
$\gz:{S_n}_0\rightarrow (-1,1)$ whose extension to $S_n$ consists of the critical manifolds
$\gz^{-1}(1)=P(\BOne\oplus 0)$ and $\gz^{-1}(-1)=P(0 \oplus L_n)$.
Letting $\theta$ be a connection one form for the Hermitian metric on ${S_n}_0$, with curvature
$d\theta = \omega_{N_n}$, an admissible K\"ahler metric and form are
given up to scale by the respective formulas
\begin{equation}\label{g}
g=\frac{1+r\gz}{r}g_{N_n}+\frac {d\gz^2}
{\Theta (\gz)}+\Theta (\gz)\theta^2,\quad
\omega = \frac{1+r\gz}{r}\omega_{N_n} + d\gz \wedge
\theta,
\end{equation}
valid on ${S_n}_0$. Here $\Theta$ is a smooth function with domain containing
$(-1,1)$ and $r$, is a real number of the same sign as
$g_{N_n}$ and satisfying $0 < |r| < 1$. The function $\gz$ is the moment 
map on $S_n$ for the circle action, decomposing $S_n$ into 
the free orbits ${S_n}_0 = \gz^{-1}((-1,1))$ and the special orbits 
$D_1= \gz^{-1}(1)$ and $D_2=\gz^{-1}(-1)$. Here generally the ruled manifold $S_n$ has an added orbifold structure given by the log pair $(S_n,\grD)$ where $\grD$ is the branch divisor
\begin{equation}\label{branchdiv2}
\grD=(1-\frac{1}{m_1})D_1+ (1-\frac{1}{m_2})D_2,
\end{equation}
with ramification indices $m_1,m_2$. Moreover, the function $\Theta$ must satisfy certain boundary conditions.
This data associated with $(S_n,\grD)$ will be called {\it admissible data}, and the function $\Theta$ is called an {\it admissible momentum profile}. We summarize by


\begin{proposition}\label{adKahprop}
Given admissible data any choice of smooth function $\Theta(\gz)$ such that
\[
\begin{array}{l}
\Theta(\gz) > 0, \quad -1 < \gz <1,\\
\\
 \Theta(\pm 1) = 0,\\
 \\
 \Theta'(-1) = 2/m_2\quad \quad \Theta'(1)=-2/m_1,
 \end{array}
\]
determines an admissible metric in the corresponding admissible K\"ahler class.
\end{proposition}

As follows from \cite{ApCaGa06} we now have the following useful proposition.
\begin{proposition}\label{extremalcond}
Given admissible data and a smooth function $\Theta(\gz)=F(\gz)/(1 + r \gz)^{d_{N}}$ satisfying the conditions of Proposition \ref{adKahprop},
the admissible metric associated with $\Theta(\gz)$ is extremal if and only if
\begin{equation}\label{extremalode}
F''(\gz) = (1+r \gz)^{d_N-1}(2d_Ns_{N_n} r + (\alpha\gz +\beta)(1+r\gz)),
\end{equation}
where $\alpha$ and $\beta$ are constants. 
\end{proposition}

The ODE \eqref{extremalode} together with the endpoint conditions of Proposition \ref{adKahprop} determines a unique solution
$F_{\mathbf r}(\gz)$. We call this the {\it extremal polynomial} for $ \Omega_{\mathbf r} $. 
The positivity condition of Proposition \ref{adKahprop}  follows automatically if we assume that $(N,\gro_N,g_N)$ has non-negative scalar curvature, but in general it may or may not hold.

In the trivial orbifold case, the following stronger theorem has been established:

\begin{theorem}\cite{ACGT08}
Let $S_n$
  be any admissible manifold and  $\Omega_{\mathbf r} $ be
  any admissible K\"ahler class on $M$. Then, $ \Omega_{\mathbf r} $ contains an extremal K\"ahler metric
   if and only if the extremal polynomial $F_{\mathbf r}(\gz)$ is
  {\rm(}strictly{\rm)} positive on ${\rm (-1, 1)}$.  In that case, there is an admissible extremal K\"ahler metric in $ \Omega_{\mathbf r} $. 
 \end{theorem}
 Unfortunately we are still lacking a generalization of this theorem to the orbifold case where
 $(m_1,m_2) \neq (1,1)$. Of course, the `if' part holds in the orbifold case as well; however, in the manifold case the `only if' part relies on the uniqueness theorem of Chen and Tian \cite{ChTi05} which is still unknown for orbifolds.

Notice from \eqref{extremalode} that the degree of $F_{\mathbf r}(\gz)$ is at most $d_N+3$.  In fact, if the degree of $F_{\mathbf r}(\gz)$ is $d_N+2$ or less, then $F_{\mathbf r}(\gz)$ is automatically
  {\rm(}strictly{\rm)} positive on ${\rm (-1, 1)}$ and from \cite{ApCaGa06} we have that $F_{\mathbf r}(\gz)$ defines an admissible CSC K\"ahler metric in 
  $ \Omega_{\mathbf r}$. This results in the following convenient fact

\begin{proposition}\cite{BoTo14a}
The existence of an admissible CSC K\"ahler metric in $\Omega_{\mathbf r}$ on the log pair $(S_n,\Delta)$ is
equivalent to
\begin{equation}\label{candkeqn}
\frac{2s_{N_n}\left( (1+r)^{d_N+1} - (1-r)^{d_N+1}\right)}{r(d_N+1)} - \frac{k\left( (1+r)^{d_N+2} - (1-r)^{d_N+2}\right)}{r^2(d_N+1)(d_N+2)} + 2c=0,
\end{equation}
where
\begin{equation}\label{whatcis}
c=  \frac{2 \left(1-r^2\right)^{d_{N}} (m_2 (1-r)+m_1 (1+r) -  2m_1 m_2 s_{N_n})}{m_1 m_2 \left((1+r)^{d_{N}+1}-(1-r)^{d_{N}+1}\right)}
\end{equation}
and 
\begin{equation}\label{whatkis}
k= \frac{2 (d_N+1) r \left(m_2 (1+r)^{d_N} (1+m_1 s_{N_n})-m_1 (1-r)^{d_N} (-1+m_2 s_{N_n})\right)}{m_1m_2 \left((1+r)^{d_N+1}-(1-r)^{d_N+1}\right)}.
\end{equation}
 \end{proposition}

Likewise, using the K\"ahler Einstein criterion from \cite{ApCaGa06}, one can establish the following.
\begin{proposition}\cite{BoTo13b}
The existence of an admissible K\"ahler Einstein metric in $\Omega_{\mathbf r}$ on the log pair $(S_n,\Delta)$ is
equivalent to $(N,\gro_N,g_N)$ being positive K\"ahler Einstein with Fano index $\cali_N$
(hence $s_{N_n}=\cali_N/n$),
\begin{equation}\label{fanoclass}
2r\cali_N/n = (1+r)/m_2 + (1-r)/m_1,
\end{equation}
and 
\begin{equation}\label{KEintegral}
\int_{-1}^1 \left((1-\gz)/m_2 -(1+\gz)/m_1\right) (1 + r \gz)^{d_{N}} d\gz = 0.
\end{equation}
\end{proposition}

\begin{remark}
A similar statement (see Section 5.2 of \cite{BoTo13b}) provides the criterion for the existence of admissible K\"ahler Ricci solitons.
\end{remark}

\section{The Sasaki Join and Admissible Sasakian Structures}\label{sasadm}
Our method involves the Sasaki join of a regular $(2p+1)$-dimensional Sasakian manifold $M$ with the weighted Sasaki 3-sphere $S^3_\bfw$ where the weight vector $\bfw=(w_1,w_2)$ has relatively prime components $w_i\in \bbz^+$ that are ordered as $w_1\geq w_2$. We shall also assume that $M$ has a regular Sasaki metric with constant scalar curvature, and that $\bfl=(l_1,l_2)$ has components $l_i$ that are positive integers satisfying the {\it admissibility condition} 
\begin{equation}\label{adcond}
\gcd(l_2,l_1w_1w_2)=1). 
\end{equation}
The join is then constructed from the following  commutative diagram
\begin{equation}\label{s2comdia}
\begin{matrix}  M\times S^3_\bfw &&& \\
                          &\searrow\pi_L && \\
                          \decdnar{\pi_{2}} && M_{\bfl,\bfw} &\\
                          &\swarrow\pi_1 && \\
                         N\times\bbc\bbp^1[\bfw] &&& 
\end{matrix}
\end{equation}
where the $\pi_2$ is the product of the standard Sasakian projections $\pi_M:M\ra{1.6}N$ and $S^3_\bfw\ra{1.6} \bbc\bbp^1[\bfw]$. The circle projection $\pi_L$ is generated by the vector field 
\begin{equation}\label{Lvec}
L_{\bfl,\bfw}=\frac{1}{2l_1}\xi_1-\frac{1}{2l_2}\xi_\bfw,
\end{equation}
and its quotient manifold $M_{\bfl,\bfw}$, which is called the $(l_1,l_2)$-join of $M$ and $S^3_\bfw$ and denoted by $M_{\bfl,\bfw}=M\star_{l_1,l_2}
S^3_\bfw$, has a naturally induced quasi-regular Sasakian structure $\cals_{\bfl,\bfw}$ with contact 1-form $\eta_{l_1,l_2,\bfw}$. It is reducible in the sense that both the transverse metric $g^T$ and the contact bundle $\cald_{l_1,l_2,\bfw}=\ker\eta_{l_1,l_2,\bfw}$ split as direct sums. Since the CR-structure $(\cald_{l_1,l_2,\bfw},J)$ is the horizontal lift of  the complex structure on $N\times\bbc\bbp^1[\bfw]$, this splits as well. The choice of $\bfw$ determines the transverse complex structure $J$.

\subsection{The $\bfw$-Sasaki Cone}
Let $\gt_{1}, \gt_\bfw$ denote the Lie algebras of the maximal tori in the Sasakian automorphism group of the Sasakian structures on $M,S^3_\bfw$, respectively. Then the unreduced Sasaki cone of the join $M_{\bfl,\bfw}$ is defined by 
\begin{equation}\label{sascone}
\gt^+_{\bfl,\bfw}=\{R\in (\gt_{1}\oplus \gt_\bfw)/\{L_{\bfl,\bfw}\}~|~\eta_{\bfl,\bfw}(R)>0\}
\end{equation}
where $\{L_{\bfl,\bfw}\}$ is the Lie algebra generated by the vector field $L_{\bfl,\bfw}$. Choosing an $R\in \gt_{\bfl,\bfw}^+$ corresponds to choosing a different contact form, namely $\frac{\eta_{\bfl,\bfw}}{\eta_{\bfl,\bfw}(R)}$, in the underlying strictly pseudoconvex CR structure. In this case the functions $\eta_{\bfl,\bfw}(R)$ are spanned by the {\it Killing potentials}, that is, the components of the moment map associated to the Hamiltonian vector fields. Note that $\gt_{\bfl,\bfw}^+$ is a subcone of the infinite dimensional {\it Reeb cone} $\calr(\cald)$ consisting of the Reeb vector fields associated to any contact form within the contact structure \cite{Boy08}.

Actually, our method only makes use of the so-called $\bfw$-subcone whose representatives lie in the 2-dimensional Lie subalgebra $\gt_\bfw$. Such elements can be described as follows: let $\{H_1,H_2\}$ denote the standard basis for $\gt_\bfw$. Then the subcone $\gt^+_\bfw$, called the $\bfw$-Sasaki cone, is the set of elements of the form $v_1H_1+v_2H_2$ with $v_1,v_2>0$. We think of the Sasaki cone $\gt^+_{\bfl,\bfw}$ as the local moduli space of Sasakian structures associated to the CR structure $(\cald_{l_1,l_2,\bfw},J)$.
However, here we are only concerned with the $\bfw$-subcone.

\subsection{Deformations in the $\bfw$-Sasaki Cone}
Here we briefly describe how to obtain new Sasakian structures by deforming in  the Sasaki cone. We choose the Reeb vector field in $\gt^+_\bfw$ defined by $\xi_\bfv=v_1H_1+v_2H_2$ where $v_1,v_2\in\bbz^+$ are relatively prime.  We refer to \cite{BoTo13,BoTo13b,BoTo14a} for details. When we deform to an arbitrary quasi-regular Reeb field in the $\bfw$-Sasaki cone, we obtain instead of the product $T^2$ action of diagram \eqref{s2comdia} the $T^2$ action given by
\begin{equation}\label{twistT2}
(x,u;z_1,z_2)\mapsto (x,e^{il_2\theta}u;e^{i(v_1\phi-l_1w_1\theta)}z_1,e^{i(v_2\phi-l_1w_2\theta)}z_2),
\end{equation}
where $(x,u)$ denote bundle coordinates on $M$ and $(z_1,z_2)\in\bbc^2$ satisfy $|z_1|^2+|z_2|^2=1$. The quotient of $M\times S^3$ by this $T^2$ action is a complex K\"ahler orbifold $B_{l_1,l_2,\bfv,\bfw}$ which can be identified with the ruled orbifold $(S_n,\grD)$ described in Section \ref{adsec}, where $\grD$ is the branch divisor of Equation \eqref{branchdiv2}. Here the ramification indices satisfy $m_i=v_i\frac{l_2}{s}=v_im$ and  $s=\gcd(|w_1v_2-w_2v_1|,l_2)$ and $n=l_1\bigl(\frac{w_1v_2-w_2v_1}{s}\bigr)$. The fiber of the orbifold $(S_n,\grD)$ is the orbifold $\bbc\bbp[v_1,v_2]/\bbz_m$. We can divide this $T^2$ action into two $S^1$ actions, namely, the first $S^1$ action is that generated by the vector field \eqref{Lvec}, and the second that generated by the Reeb vector field $\xi_\bfv\in\gt^+_\bfw$ with $v_1,v_2$ relatively prime positive integers. This gives rise to the commutative diagram
\begin{equation}\label{comdia1}
\begin{matrix}  M\times S^{3}_\bfw &&& \\
                          &\searrow && \\
                          \decdnar{\pi_B} && M_{\bfl,\bfw} &\\
                          &\swarrow{\pi_\bfv} && \\
                          B_{l_1,l_2,\bfv,\bfw} &&& 
\end{matrix}
\end{equation}
where $\pi_B$ and $\pi_\bfv$ are the obvious projections.

Each $\bfw$-Sasaki cone has a unique ray, called {\it almost regular}, obtained by setting $\bfv=(v_1,v_2)=(1,1)$. In this case the fiber of the orbibundle is $\bbc\bbp^1/\bbz_m$, and if $s=l_2$ the Reeb vector field is regular in which case there is no branch divisor and $(S_n,\emptyset)$ has a trivial orbifold structure. Note that in the latter case $l_2$ must divide $w_1-w_2$. Hence, for $l_2=1$ every $\bfw$-Sasaki cone has a unique regular ray of Reeb vector fields.

The induced K\"ahler class on $B_{l_1,l_2,\bfv,\bfw}= (S_n,\grD)$ can now be determined to be a multiple of the admissible class $\Omega_r$ from \eqref{admKahclass}, where
$r=\frac{w_1v_2-w_2v_1}{w_1v_2+w_2v_1}$ \cite{BoTo14a}. Therefore the constructions of admissible K\"ahler metrics produces quasi-regular Sasaki metrics in the $\bfw$-cone of the joins $M_{\bfl,\bfw}$. Moreover these constructions can be extended to the irregular case as well. We refer to the paper \cite{BoTo14a} for the details.

\section{The Topology of the Joins}
Here we discuss some general results concerning the topology of the joins $M_{\bfl,\bfw}$. Of course, as we shall see we can say much more in certain specific cases. First we easily see from the long exact homotopy sequence of the $S^1$-bundle $\pi_L$ of Diagram \eqref{s2comdia} that 
\begin{lemma}\label{gentop}
Let $M_{\bfl,\bfw}$ be as described above. Then
\begin{enumerate}
\item The natural map $\pi_1(M)\ra{1.8} \pi_1(M_{\bfl,\bfw})$ is surjective. In particular, if $M$ is simply connected so is $M_{\bfl,\bfw}$.
\item If $M$ is simply connected, then $\pi_2(M_{\bfl,\bfw})\approx \pi_2(M)\oplus \bbz$.
\item For $i\geq 3$, we have $\pi_i(M_{\bfl,\bfw})\approx \pi_i(M)\oplus \pi_i(S^3)$.
\end{enumerate}
\end{lemma}

\subsection{The Contact Structure}
The first Chern class of the complex vector bundle $\cald_{l_1,l_2,\bfw}$ is an important contact invariant. For our joins we have
\begin{equation}\label{c1cald}
c_1(\cald_{l_1,l_2,\bfw})=\pi_M^*c_1(N)-l_1|\bfw|\grg,
\end{equation}
where $\grg$ is a generator in $H^2(M_{\bfl,\bfw},\bbz)$.
We are particularly interested in a special case. Let $\gro_N$ denote the K\"ahler form on $N$. We say that the class $[\gro_N]$ is {\it quasi-monotone} if  $c_1(N)=\cali_N[\gro_N]$ for some integer $\cali_N$. Here $\cali_N$ is the {\it Fano index} when $\cali_N$ is positive (the monotone case) and the {\it canonical index} when it is negative. We also allow the case $\cali_N=0$. So when $[\gro_N]$ is quasi-monotone we have
\begin{equation}\label{c1cald2}
c_1(\cald_{l_1,l_2,\bfw})=(l_2\cali_N -l_1|\bfw|)\grg.  
\end{equation}

Of course, from this one can extract a topological invariant, namely, the second Stiefel-Whitney class $w_2(M_{\bfl,\bfw})$ which is the mod 2 reduction of $c_1(\cald_{l_1,l_2,\bfw})$. In the quasi-monotone case we have 
\begin{equation}\label{w2}
w_2(M_{\bfl,\bfw})\equiv (l_2\cali_N-l_1|\bfw|) \mod 2.
\end{equation}

\subsection{Computing the Cohomology Ring}
More importantly there is a method, used in \cite{WaZi90,BG00a} (see also Section 7.6.2 of \cite{BG05}) for computing the cohomology ring of $M_{\bfl,\bfw}$. Since there are orbifolds involved it is convenient to work with related fibrations involving classifying spaces instead of those of Diagram (\ref{s2comdia}). We thus have the commutative diagram of fibrations
\begin{equation}\label{orbifibrationexactseq}
\begin{matrix}M\times S^3_\bfw &\ra{2.6} &M_{\bfl,\bfw}&\ra{2.6}
&\mathsf{B}S^1 \\
\decdnar{=}&&\decdnar{}&&\decdnar{\psi}\\
M\times S^3_\bfw&\ra{2.6} & N\times\mathsf{B}\bbc\bbp^1[\bfw]&\ra{2.6}
&\mathsf{B}S^1\times \mathsf{B}S^1\, 
\end{matrix} \qquad \qquad
\end{equation}
where $\mathsf{B}G$ is the classifying space of a group $G$ or Haefliger's classifying space \cite{Hae84} of an orbifold if $G$ is an orbifold. Note that the lower fibration is a product of fibrations. Then using 
\begin{lemma}\label{cporbcoh}
For $w_1$ and $w_2$ relatively prime positive integers we have
$$H^r_{orb}(\bbc\bbp^1[\bfw],\bbz)=H^r( \mathsf{B}\bbc\bbp^1[\bfw],\bbz)= \begin{cases}
                    \bbz &\text{for $r=0,2$,}\\                  
                    \bbz_{w_1w_2} &\text{for $r>2$ even,}\\
                     0 &\text{for $r$ odd.}
                     \end{cases}$$           
\end{lemma}
\noindent together with Diagram \eqref{orbifibrationexactseq} and standard arguments we obtain
\begin{algorithm}
Given the differentials in the spectral sequence of the fibration 
$$M\ra{1.5}N\ra{1.5}\mathsf{B}S^1,$$ 
one can use the commutative diagram (\ref{orbifibrationexactseq}) to compute the cohomology ring of the join manifold $M_{\bfl,\bfw}$.
\end{algorithm}

A case of particular interest occurs by taking $M=S^{2p+1}$ with $p>1$ \cite{BoTo14b}. The case $p=1$ behaves topologically different and is treated in Section \ref{secg0}.

\begin{theorem}\label{findiff}
In each odd dimension $2p+3>5$ there exist countably infinite simply connected toric contact manifolds $M_{\bfl,\bfw} =S^{2p+1}\star_{l_1,l_2}S^3_\bfw$ of Reeb type depending on $4$ positive integers $l_1,l_2,w_1,w_2$ satisfying $\gcd(l_2,l_1w_i)=\gcd(w_1,w_2)=1$, and with integral cohomology ring
$$H^*(M_{\bfl,\bfw},\bbz)\approx\bbz[x,y]/(w_1w_2l_1^2x^2,x^{p+1},x^2y,y^2)$$
where $x,y$ are classes of degree $2$ and $2p+1$, respectively. Furthermore, with $l_1,w_1,w_2$ fixed there are a finite number of diffeomorphism types with the given cohomology ring. Hence, in each such dimension there exist simply connected smooth manifolds with countably infinite toric contact structures of Reeb type that are inequivalent as contact structures.
\end{theorem}

\section{The Main Results}
In this section we collect the main general results that come from the construction of Section \ref{sasadm}. More can be said in the special case of dimension 5 and we will present our results on $S^3$-bundles over Riemann surfaces in Section \ref{sasoverRiem}. Our general results which appear in \cite{BoTo13b,BoTo14a,BoTo14b} prove the existence of extremal and constant scalar curvature Sasaki metrics by applying the admissible construction of Section \ref{admissec} to the join construction of Section \ref{sasadm}. 

\begin{remark}
Whenever, in the following, we state that a Sasaki structure is Einstein, is a Ricci soliton, is extremal, or has constant scalar curvature, we mean that there is a Sasaki structure in the same isotopy class with such property.
\end{remark}

\subsection{Constant Scalar Curvature Sasaki Metrics}

Our main results on constant scalar curvature were given in \cite{BoTo14a}:

\begin{theorem}\label{admjoincsc}
Let $M_{\bfl,\bfw}$ be the $S^3_\bfw$-join with a regular Sasaki manifold $M$ which is an $S^1$-bundle over a compact K\"ahler manifold $N$ with constant scalar curvature. Then for each vector $\bfw=(w_1,w_2)\in \bbz^+\times\bbz^+$ with relatively prime components satisfying $w_1>w_2$ there exists a Reeb vector field $\xi_\bfv$ in the 2-dimensional $\bfw$-Sasaki cone on $M_{\bfl,\bfw}$ such that the corresponding ray of Sasakian structures $\cals_a=(a^{-1}\xi_\bfv,a\eta_\bfv,\Phi,g_a)$ has constant scalar curvature. Suppose that in addition  the scalar curvature of N satisfies $s_N\geq 0$, then the $\bfw$-Sasaki cone is exhausted by extremal Sasaki metrics. Moreover, if the scalar curvature of N satisfies $s_N> 0$, then for sufficiently large $l_2$ there are at least three CSC rays in the $\bfw$-Sasaki cone of the join $M_{\bfl,\bfw}$.
\end{theorem}

Most often the CSC Sasaki metrics belong to irregular Sasakian structures. 

Generally, it seems difficult to say anything about the diffeomorphism types of $M_{\bfl,\bfw}$. However, when $M$ is an odd dimensional sphere we can say a bit more. Combining Theorems \ref{findiff} and \ref{admjoincsc} we obtain:

\begin{theorem}\label{findiff2}
The contact structures of Theorem \ref{findiff} admit a $p+2$ dimensional cone of Sasakian structures with a ray of constant scalar curvature Sasaki metrics and for $l_2$ large enough they admit at least $3$ such rays. Moreover, for $l_1,w_1,w_2$ fixed, there are a finite number of diffeomorphism types; hence, in each odd dimension $2p+3>5$ there are smooth manifolds with a countably infinite number of inequivalent contact structures each having at least 3 rays of CSC Sasaki metrics.
\end{theorem}

It is interesting to note the relation with the CR Yamabe problem \cite{JeLe87} which proves the existence of a `pseudohermitian' metric of constant Webster scalar curvature for any strictly pseudoconvex CR structure\footnote{a strictly pseudeconvex CR structure is a particular case of a contact structure.} $(\cald,J)$. When $(\cald,J)$ is the underlying CR structure of a Sasaki metric, the pseudohermitian metric is precisely the transverse K\"ahler metric $g^T$, and the Webster scalar curvature is precisely the scalar curvature $s^T$ of $g^T$. The CR Yamabe problem has an important invariant, the so-called {\it CR Yamabe invariant} $\grl(M)$. We refer to \cite{JeLe87} for its definition and properties. Suffice it here to say that if $M$ is not an odd dimensional sphere, $\grl(M)$ is attained by some contact form in the contact structure with constant Webster scalar curvature. Moreover, if $\grl(M)$ satisfies $\grl(M)\leq 0$, the constant scalar curvature solution is unique up to transverse homothety. So for Sasaki structures with $\grl(M)\leq 0$ the CSC ray is not only unique in the Sasaki cone, but in the entire infinite dimensional Reeb cone. However, generally the transverse homothety class of CSC metrics is not unique. Our results give examples of Sasakian structures that exhibit this lack of uniqueness for the CR Yamabe problem, and they occur in the two-dimensional subcone $\gt_\bfw^+$. When $\bfw\neq (1,1)$ there are an infinite number of CR structures with three CSC rays which are inequivalent under CR automorphisms.

\subsection{Sasaski-Einstein Metrics}
For Sasaki-Einstein manifolds we have \cite{BoTo13b}

\begin{theorem}\label{admjoinse}
Let $M_{\bfl,\bfw}$ be the $S^3_\bfw$-join with a regular Sasaki manifold $M$ which is an $S^1$-bundle over a compact positive K\"ahler-Einstein manifold $N$ with a primitive K\"ahler class $[\gro_N]\in H^2(N,\bbz)$. Assume that the relatively prime positive integers $(l_1,l_2)$ are the relative Fano indices given explicitly by 
$$l_{1}=\frac{\cali_N}{\gcd(w_1+w_2,\cali_N)},\qquad   l_2=\frac{w_1+w_2}{\gcd(w_1+w_2,\cali_N)},$$ where $\cali_N$ denotes the Fano index of $N$. Then for each vector $\bfw=(w_1,w_2)\in \bbz^+\times\bbz^+$ with relatively prime components satisfying $w_1>w_2$ there exists a Reeb vector field $\xi_\bfv$ in the 2-dimensional $\bfw$-Sasaki cone on $M_{\bfl,\bfw}$ such that the corresponding Sasakian structure $\cals=(\xi_\bfv,\eta_\bfv,\Phi,g)$ is Sasaki-Einstein. Furthermore, the Sasaki-Einstein metric is unique up to transverse holomorphic transformations. Additionally (up to isotopy) the Sasakian structure associated to every single ray, $\xi_\bfv$, in the $\bfw$-Sasaki cone is a Sasaki-Ricci soliton as well as extremal.
\end{theorem}

These Sasaki-Einstein metrics were obtained earlier by physicists in \cite{GMSW04a,GMSW04b} by a different method, and they are most often irregular. A special case of particular interest are the $Y^{p,q}$ structures on $S^2\times S^3$ which are discussed further in Example \ref{Ypqex} below. The uniqueness statement in the theorem follows from \cite{NiSe12} which proves transverse uniqueness, up to transverse holomorphic transformations, then using \cite{MaSpYau06} which proves uniqueness in the Sasaki cone.

\subsection{Extremal Sasaki Metrics}
As far as general extremal Sasaki metrics are concerned we have 
\begin{theorem}\label{extthm}
Let $M_{\bfl,\bfw}$ be the $S^3_\bfw$-join with a regular Sasaki manifold $M$ which is an $S^1$-bundle over a compact K\"ahler manifold $N$ with constant scalar curvature $s_N\geq 0$. Then the $\bfw$-Sasaki cone is exhausted by extremal Sasaki metrics.
\end{theorem}

In particular, if the K\"ahler manifold $N$ has no Hamiltonian vector fields, then the $\bfw$-Sasaki cone is the entire Sasaki cone in which case the entire Sasaki cone is exhausted by extremal Sasaki metrics. A particular example is when $N$ is an algebraic K3 surface in which case there are many choices of complex structures and many choices of line bundles. In all cases $M=21\# (S^2\times S^3)$. It is interesting to contemplate the possible diffeomorphism types of the 7-manifolds $M_{\bfl,\bfw}=21\#(S^2\times S^3)\star_{l_1,l_2} S^3_\bfw$.

When $s_N<0$ there can be obstructions to the existence of extremal Sasaki metrics as we shall see explicitly in the next section.

\section{Sasakian Geometry of $S^3$-bundles over Riemann Surfaces}\label{sasoverRiem}
Much can be said when $\dim_\bbr N=2$, that is when $N$ is a Riemann surface $\grS_g$ of genus $g$. In this case there is a classification, and each $\grS_g$ admits a constant scalar curvature K\"ahler metric. It is well known that there are exactly two such bundles up to diffeomorphism, the trivial bundle $\grS_g\times S^3$ and the non-trivial bundle $\grS_g\tilde{\times} S^3$. Although there are some important differences when the Riemann surface has genus $g=0$, we first describe the properties they have in common. We should, however, note that historically the two cases were treated separately. An important difference that we should mention is that for $g>0$ we must have $l_2=1$ in order to have an $S^3$-bundle over a Riemann surface. It is precisely in this case that there are similarities between the $g=0$ and $g\neq 0$ cases. When $g>0$ the case $l_2>1$ was studied in \cite{Cas14}.

\subsection{Genus $g=0$}\label{secg0}
Toric contact structures of Reeb type, hence Sasakian, on $S^3$-bundles over $S^2$ were studied extensively in \cite{BG00b,Ler02a,Boy11,Leg10,BoPa10} among others.   In this case the Sasaki cone has dimension three. We also emphasize at this stage that the admissibility condition \eqref{adcond} must hold.
The following is a restatement of Theorem 2 of \cite{BoPa10}.

\begin{theorem}\label{BoPathm}
The contact structures $\cald_{l_1,l_2,\bfw}$ and $\cald_{l'_1,l_2,\bfw'}$ are contactomorphic if the following conditions hold
$$l_1|\bfw|=l'_1|\bfw'|, \qquad \gcd(l_2,l_1(w_1-w_2))=\gcd(l_2,l'_1(w'_1-w'_2)).$$
\end{theorem}

Notice that when $l_2=1$ we have (which was missed in \cite{BoPa10})

\begin{corollary}\label{l2=1}
The contact structures $\cald_{l_1,1,\bfw}$ and $\cald_{l'_1,1,\bfw'}$ are contactomorphic if and only if $l_1|\bfw|=l'_1|\bfw'|$, i.e. if and only their first Chern classes are equal.
\end{corollary}

There is another result that follows from Theorem \ref{BoPathm} together with Theorems 4.2 and 4.8\footnote{There is a typo in part (2) of Theorem 4.8 of \cite{BoPa10}. $2k-j+1$ should be $2k-2j+1$.} of \cite{BoPa10}. First, we mention that the Reeb vector fields in these theorems are almost regular, and then we recall that part (2) of these theorems says that the orbifold structure of $(S_n,\grD)$ is trivial if and only if $l_2$ divides $w_1-w_2$. This implies
\begin{corollary}\label{regcor}
If any $\bfw$-cone of a bouquet has a regular Reeb field then all $\bfw$-cones of the bouquet have a regular Reeb field.
\end{corollary}
We call a bouquet with a regular Reeb vector field a {\it regular bouquet}. There are many regular bouquets in $S^3$-bundles over $S^2$. For examples of regular bouquets see Section 3 of \cite{Boy10b}.

The phenomenon of the existence of multiple CSC rays in the Sasaki cone on $S^3$-bundles over $S^2$ was discovered by Legendre \cite{Leg10}. Here we give a bound for the existence of multiple CSC rays in the $\bfw$-Sasaki subcone. 

\begin{proposition}\label{multcsc}
Consider the toric contact structure $\cald_{l_1,l_2,\bfw}$ with $\bfw\neq (1,1)$ on an $S^3$-bundle over $S^2$. If the inequality 
$$2l_2>16l_1w_1-5l_1w_2$$
holds, then there are at least three CSC Sasaki metrics in the $\bfw$-Sasaki subcone.
\end{proposition}

\begin{proof}
The proof follows from the analysis of Example 6.7 in \cite{BoTo14b}.
\end{proof}

When $\bfw=(1,1)$ there can also be multiple CSC rays, but two of them are equivalent under the action of the Weyl group of the CR automorphism group. In this case our bound for multiple rays is $2l_2>11l_1$.

\begin{example}\label{4bouq}
Our first example is a $4$-bouquet on $S^2\times S^3$ when $l_2=1$. In this case the orbifold structure is trivial, viz. $(S_{2m},\emptyset)$ where $S_{2m}$ is an even Hirzebruch surface with $m=\frac{1}{2}l_1(w_1-w_2)$. So this is a regular bouquet.
\bigskip

\centerline{$4$-bouquet on $S^2\times S^3$}
\begin{center}
\begin{tabular}{| l | l | l |}
\hline
$m$ & $l_1$ & \bfw \\ \hline
0 & 4 & (1,1) \\ \hline
1 & 1 & (5,3) \\ \hline
2 & 2 & (3,1) \\ \hline
3 & 1 & (7,1) \\ \hline
\end{tabular}
\end{center}
\medskip

However, the $4$-bouquet does not persist for arbitrary $l_2$. First, we need to satisfy the admissibility condition \eqref{adcond}. Thus, to satisfy condition \eqref{adcond} for all 4 contact structures we need $\gcd(l_2,210)=1$ in which case $l_2$ cannot divide $w_1-w_2$. So if we choose such an $l_2\neq 1$ we see that the three members with $m=1,2,3$ form a non-regular 3-bouquet which cannot include $m=0$. So for such an $l_2>1$ the 3-bouquet has only non-trivial orbifold quotients, $(S_{2m},\grD)$ for $m=1,2,3$ where the ramification index of both divisors is $l_2$ for all three quotients. Furthermore, by Theorem \ref{admjoincsc} for $l_2$ sufficiently large, there exist 3 CSC rays in each of the 3 Sasaki cones of the bouquet. In all cases the contact structure $\cald$ is determined by $l_2$ since its first Chern class is $c_1(\cald)=2l_2-8$. In all cases the symplectic form on the base orbifold is $\grs_0+4\grs_F$ where $\grs_0,\grs_F$ are the standard area forms on the base $S^2$ and fiber $S^2$, respectively.

From Proposition \ref{multcsc} we see that if we choose $l_2>53$ all three members of the 3-bouquet will have three CSC Sasaki metrics.
\end{example}

\begin{example}\label{Ypqex}
This case was treated in \cite{Boy11,BoPa10,BoTo13b}. It is the $Y^{p,q}$ manifolds diffeomorphic to $S^2\times S^3$ discovered in \cite{GMSW04a}, where $p$ and $q$ are relatively prime integers satisfying $1\leq q<p$. The relation between the pair $(p,q)$ and our notation is $l_2=p$ and  $l_1\bfw=(p+q,p-q)$. Applying Theorem \ref{BoPathm} to the $Y^{p,q}$'s we see that $Y^{p,q}$ is contactomorphic to $Y^{p',q'}$ if
$p'=p$ and $\gcd(p',2q')=\gcd(p,2q)$. But since $p$ and $q$ are relatively prime we conclude that $Y^{p',q'}$ is contactomorphic to $Y^{p,q}$ if $p=p'$. So if for each integer $p>1$ we let $\phi(p)$ denote the Euler phi function, that is, the number of positive integers $q$ that are less than $p$ and relatively prime to $p$, then there is a $\phi(p)$-bouquet $\gB_{\phi(p)}$ of Sasaki cones on $S^2\times S^3$ such that each Sasaki cone has a unique Reeb vector field with a Sasaki-Einstein metric in its $\bfw$ subcone. Uniqueness follows from \cite{CFO07}.  Also for each Sasaki cone each element of the $\bfw$ subcone has, up to contact isotopy, both an extremal Sasaki metric and a Sasaki-Ricci soliton.
\end{example}

\subsection{Genus $g>0$}
The existence of CSC Sasaki metrics follows from Theorem \ref{admjoincsc}; however,  this case was studied in detail earlier in \cite{BoTo11,BoTo13}.  In particular, when $g>0$ there is a unique ray of admissible constant scalar curvature Sasaki metrics in each 2-dimensional Sasaki cone. But also in this case the bouquet phenomenon occurs. 

Before stating our general theorem, we revisit Example \ref{4bouq}:

\begin{example}[Example \ref{4bouq} revisited]\label{ex7.5rev}
When $g>0$ we must take $l_2=1$ to obtain $\grS_g\times S^3$ \cite{Cas14}, and for this case we recapture the 4-bouquet of Example \ref{4bouq} for all genera $g$.
\end{example}

More generally we have

\begin{theorem}\label{g>0thm}
Let $M$ be the total space of an $S^3$-bundle over a Riemann surface $\grS_g$ of genus $g>0$. Then for each $k\in \bbz^+$ the 5-manifold $M$ admits a contact structure $\cald_k$ of Sasaki type consisting of $k$ 2-dimensional Sasaki cones $\grk(\cald_k,J_m)$ labelled by $m=0,\ldots,k-1$ each of which admits a unique ray of Sasaki metrics of constant scalar curvature such that the transverse K\"ahler structure admits a Hamilitonian 2-form. Moreover, if $M$ is the trivial bundle $\grS_g\times S^3$, there is a $k+1$-bouquet $\gB_{k+1}(\cald_k)$, consisting of $k$ 2-dimensional Sasaki cones and one 1-dimensional Sasaki cone on each contact structure $\cald_k$. 
\end{theorem}

The $k$ 2-dimensional Sasaki cones on $\grS_g\times S^3$ are inequivalent as $S^1$-equivariant contact structures. The 2-tori associated to the $\bfw$-Sasaki cones belong to distinct conjugacy classes of maximal tori in the contactomorphism group $\gC\go\gn(\cald_k)$. This is shown by computing equivariant Gromov-Witten invariants.

We know that for $g\leq 4$ the entire 2-dimensional $\bfw$-cones are exhausted by extremal Sasaki metrics; however, there are also some non-existence results in this case as in \cite{To-Fr98}. 
\begin{theorem}\label{deg>0intro}
For any choice of genus $g=20,21,...$ there exist at least one choice of $(k,m)$ with $m=1,\ldots,k-1$ such that the regular ray in the Sasaki cone $\grk(\cald_k,J_m)$ admits no extremal representative. 
\end{theorem}

\def\cprime{$'$} \def\cprime{$'$} \def\cprime{$'$} \def\cprime{$'$}
  \def\cprime{$'$} \def\cprime{$'$} \def\cprime{$'$} \def\cprime{$'$}
  \def\cdprime{$''$} \def\cprime{$'$} \def\cprime{$'$} \def\cprime{$'$}
  \def\cprime{$'$}
\providecommand{\bysame}{\leavevmode\hbox to3em{\hrulefill}\thinspace}
\providecommand{\MR}{\relax\ifhmode\unskip\space\fi MR }
\providecommand{\MRhref}[2]{%
  \href{http://www.ams.org/mathscinet-getitem?mr=#1}{#2}
}
\providecommand{\href}[2]{#2}


\begin{thebibliography}{ACGTF08b}

\bibitem[ACG03]{ApCaGa03}
V.~Apostolov, D.~M.~J. Calderbank, and P.~Gauduchon, \emph{The geometry of
  weakly self-dual {K}\"ahler surfaces}, Compositio Math. \textbf{135} (2003),
  no.~3, 279--322. \MR{1956815 (2004f:53045)}

\bibitem[ACG06]{ApCaGa06}
Vestislav Apostolov, David M.~J. Calderbank, and Paul Gauduchon,
  \emph{Hamiltonian 2-forms in {K}\"ahler geometry. {I}. {G}eneral theory}, J.
  Differential Geom. \textbf{73} (2006), no.~3, 359--412. \MR{2228318
  (2007b:53149)}

\bibitem[ACGTF04]{ACGT04}
V.~Apostolov, D.~M.~J. Calderbank, P.~Gauduchon, and C.~W.
  T{\o}nnesen-Friedman, \emph{Hamiltonian $2$-forms in {K}\"ahler geometry.
  {II}. {G}lobal classification}, J. Differential Geom. \textbf{68} (2004),
  no.~2, 277--345. \MR{2144249}

\bibitem[ACGTF08a]{ACGT08c}
Vestislav Apostolov, David M.~J. Calderbank, Paul Gauduchon, and Christina~W.
  T{\o}nnesen-Friedman, \emph{Extremal {K}\"ahler metrics on ruled manifolds
  and stability}, Ast\'erisque (2008), no.~322, 93--150, G{\'e}om{\'e}trie
  diff{\'e}rentielle, physique math{\'e}matique, math{\'e}matiques et
  soci{\'e}t{\'e}. II. \MR{2521655 (2010h:32029)}

\bibitem[ACGTF08b]{ACGT08}
\bysame, \emph{Hamiltonian 2-forms in {K}\"ahler geometry. {III}. {E}xtremal
  metrics and stability}, Invent. Math. \textbf{173} (2008), no.~3, 547--601.
  \MR{MR2425136 (2009m:32043)}

\bibitem[BG00a]{BG00a}
C.~P. Boyer and K.~Galicki, \emph{On {S}asakian-{E}instein geometry}, Internat.
  J. Math. \textbf{11} (2000), no.~7, 873--909. \MR{2001k:53081}

\bibitem[BG00b]{BG00b}
Charles~P. Boyer and Krzysztof Galicki, \emph{A note on toric contact
  geometry}, J. Geom. Phys. \textbf{35} (2000), no.~4, 288--298. \MR{MR1780757
  (2001h:53124)}

\bibitem[BG08]{BG05}
\bysame, \emph{Sasakian geometry}, Oxford Mathematical Monographs, Oxford
  University Press, Oxford, 2008. \MR{MR2382957 (2009c:53058)}

\bibitem[BGO07]{BGO06}
Charles~P. Boyer, Krzysztof Galicki, and Liviu Ornea, \emph{Constructions in
  {S}asakian geometry}, Math. Z. \textbf{257} (2007), no.~4, 907--924.
  \MR{MR2342558 (2008m:53103)}

\bibitem[BGS08]{BGS06}
Charles~P. Boyer, Krzysztof Galicki, and Santiago~R. Simanca, \emph{Canonical
  {S}asakian metrics}, Commun. Math. Phys. \textbf{279} (2008), no.~3,
  705--733. \MR{MR2386725}

\bibitem[Boy86]{Boy86}
C.~P. Boyer, \emph{Conformal duality and compact complex surfaces}, Math. Ann.
  \textbf{274} (1986), no.~3, 517--526. \MR{87i:53068}

\bibitem[Boy88a]{Boy88}
\bysame, \emph{A note on hyper-{H}ermitian four-manifolds}, Proc. Amer. Math.
  Soc. \textbf{102} (1988), no.~1, 157--164. \MR{915736 (89c:53049)}

\bibitem[Boy88b]{Boy88b}
\bysame, \emph{Self-dual and anti-self-dual {H}ermitian metrics on compact
  complex surfaces}, Mathematics and general relativity (Santa Cruz, CA, 1986),
  Contemp. Math., vol.~71, Amer. Math. Soc., Providence, RI, 1988,
  pp.~105--114. \MR{954411 (89h:53127)}

\bibitem[Boy08]{Boy08}
Charles~P. Boyer, \emph{Sasakian geometry: the recent work of {K}rzysztof
  {G}alicki}, Note Mat. \textbf{28} (2008), no.~[2009 on cover], suppl. 1,
  63--105 (2009). \MR{2640576}

\bibitem[Boy11a]{Boy11}
\bysame, \emph{Completely integrable contact {H}amiltonian systems and toric
  contact structures on {$S^2\times S^3$}}, SIGMA Symmetry Integrability Geom.
  Methods Appl. \textbf{7} (2011), Paper 058, 22. \MR{2861218}

\bibitem[Boy11b]{Boy10b}
\bysame, \emph{Extremal {S}asakian metrics on {$S^3$}-bundles over {$S^2$}},
  Math. Res. Lett. \textbf{18} (2011), no.~1, 181--189. \MR{2756009
  (2012d:53132)}

\bibitem[Boy13]{Boy10a}
\bysame, \emph{Maximal tori in contactomorphism groups}, Differential Geom.
  Appl. \textbf{31} (2013), no.~2, 190--216. \MR{3032643}

\bibitem[BP14]{BoPa10}
Charles~P. Boyer and Justin Pati, \emph{On the equivalence problem for toric
  contact structures on ${S}^3$-bundles over ${S}^2$}, Pac. Jour. of Math.
  \textbf{267} (2014), no.~2, 277--324.

\bibitem[Bry01]{Bry01}
R.~L. Bryant, \emph{Bochner-{K}\"ahler metrics}, J. Amer. Math. Soc.
  \textbf{14} (2001), no.~3, 623--715 (electronic). \MR{1824987 (2002i:53096)}

\bibitem[BTF13a]{BoTo11}
Charles~P. Boyer and Christina~W. T{\o}nnesen-Friedman, \emph{Extremal
  {S}asakian geometry on {$T^2\times S^3$} and related manifolds}, Compos.
  Math. \textbf{149} (2013), no.~8, 1431--1456. \MR{3103072}

\bibitem[BTF13b]{BoTo13b}
\bysame, \emph{The {S}asaki join, {H}amiltonian 2-forms, and
  {S}asaki-{E}instein metrics}, preprint; arXiv:1309.7067 [math.DG] (2013).

\bibitem[BTF13c]{BoTo12b}
\bysame, \emph{Sasakian manifolds with perfect fundamental groups}, Afr.
  Diaspora J. Math. \textbf{14} (2013), no.~2, 98--117. \MR{3093238}

\bibitem[BTF14a]{BoTo13}
\bysame, \emph{Extremal {S}asakian geometry on ${ S}^3$-bundles over {R}iemann
  surfaces}, Int. Math. Res. Not. IMRN (2014), doi:10.1093/139.

\bibitem[BTF14b]{BoTo14a}
\bysame, \emph{The {S}asaki join, {H}amiltonian 2-forms, and constant scalar
  curvature}, preprint; arXiv:1402.2546 Math.DG (2014).

\bibitem[BTF14c]{BoTo14b}
\bysame, \emph{Simply connected manifolds with infinitely many toric contact
  structures and constant scalar curvature {S}asaki metrics}, preprint;
  arXiv:1404.3999 (2014).

\bibitem[BW58]{BoWa}
W.~M. Boothby and H.~C. Wang, \emph{On contact manifolds}, Ann. of Math. (2)
  \textbf{68} (1958), 721--734. \MR{22 \#3015}

\bibitem[Cal56]{Cal56}
E.~Calabi, \emph{The space of {K}\"ahler metrics}, Proceedings of the
  International Congress of Mathematicians, Vol. 1, 2 (Amsterdam 2, 1954)
  (Amsterdam), North-Holland, 1956, pp.~206--207.

\bibitem[Cal82]{Cal82}
\bysame, \emph{Extremal {K}\"ahler metrics}, Seminar on Differential Geometry,
  Ann. of Math. Stud., vol. 102, Princeton Univ. Press, Princeton, N.J., 1982,
  pp.~259--290. \MR{83i:53088}

\bibitem[Cas14]{Cas14}
Candelario Casta{\~n}eda, \emph{Sasakian geometry of lens space bundles over
  {R}iemann surfaces}, University of {N}ew {M}exico Thesis (2014).

\bibitem[CFO08]{CFO07}
Koji Cho, Akito Futaki, and Hajime Ono, \emph{Uniqueness and examples of
  compact toric {S}asaki-{E}instein metrics}, Comm. Math. Phys. \textbf{277}
  (2008), no.~2, 439--458. \MR{MR2358291}

\bibitem[CT05]{ChTi05}
Xiuxiong Chen and Gang Tian, \emph{Uniqueness of extremal {K}\"ahler metrics},
  C. R. Math. Acad. Sci. Paris \textbf{340} (2005), no.~4, 287--290.
  \MR{2121892 (2006h:32020)}

\bibitem[Gau77a]{Gau77b}
Paul Gauduchon, \emph{Fibr\'es hermitiens \`a endomorphisme de {R}icci non
  n\'egatif}, Bull. Soc. Math. France \textbf{105} (1977), no.~2, 113--140.
  \MR{0486672 (58 \#6375)}

\bibitem[Gau77b]{Gau77a}
\bysame, \emph{Le th\'eor\`eme de l'excentricit\'e nulle}, C. R. Acad. Sci.
  Paris S\'er. A-B \textbf{285} (1977), no.~5, A387--A390. \MR{0470920 (57
  \#10664)}

\bibitem[Gau10]{Gau09b}
\bysame, \emph{Calabi's extremal {K}\"ahler metrics}, preliminary version,
  2010.

\bibitem[GMSW04a]{GMSW04b}
J.~P. Gauntlett, D.~Martelli, J.~Sparks, and D.~Waldram, \emph{A new infinite
  class of {S}asaki-{E}instein manifolds}, Adv. Theor. Math. Phys. \textbf{8}
  (2004), no.~6, 987--1000. \MR{2194373}

\bibitem[GMSW04b]{GMSW04a}
\bysame, \emph{Sasaki-{E}instein metrics on {$S^2\times S^3$}}, Adv. Theor.
  Math. Phys. \textbf{8} (2004), no.~4, 711--734. \MR{2141499}

\bibitem[Gua95]{Gua95}
Daniel Guan, \emph{Existence of extremal metrics on compact almost homogeneous
  {K}\"ahler manifolds with two ends}, Trans. Amer. Math. Soc. \textbf{347}
  (1995), no.~6, 2255--2262. \MR{1285992 (96a:58059)}

\bibitem[Hae84]{Hae84}
A.~Haefliger, \emph{Groupo\"\i des d'holonomie et classifiants}, Ast\'erisque
  (1984), no.~116, 70--97, Transversal structure of foliations (Toulouse,
  1982). \MR{86c:57026a}

\bibitem[HS02]{HwaSi02}
Andrew~D. Hwang and Michael~A. Singer, \emph{A momentum construction for
  circle-invariant {K}\"ahler metrics}, Trans. Amer. Math. Soc. \textbf{354}
  (2002), no.~6, 2285--2325 (electronic). \MR{1885653 (2002m:53057)}

\bibitem[Hwa94]{Hwa94}
Andrew~D. Hwang, \emph{On existence of {K}\"ahler metrics with constant scalar
  curvature}, Osaka J. Math. \textbf{31} (1994), no.~3, 561--595. \MR{1309403
  (96a:53061)}

\bibitem[JL87]{JeLe87}
David Jerison and John~M. Lee, \emph{The {Y}amabe problem on {CR} manifolds},
  J. Differential Geom. \textbf{25} (1987), no.~2, 167--197. \MR{MR880182
  (88i:58162)}

\bibitem[KS88]{KoSa88}
Norihito Koiso and Yusuke Sakane, \emph{Nonhomogeneous {K}\"ahler-{E}instein
  metrics on compact complex manifolds. {II}}, Osaka J. Math. \textbf{25}
  (1988), no.~4, 933--959. \MR{983813 (90e:53061)}

\bibitem[LeB91]{LeB91b}
C.~LeBrun, \emph{Scalar-flat {K}\"ahler metrics on blown-up ruled surfaces}, J.
  Reine Angew. Math. \textbf{420} (1991), 161--177. \MR{1124569 (92i:53066)}

\bibitem[Leg11]{Leg10}
Eveline Legendre, \emph{Existence and non-uniqueness of constant scalar
  curvature toric {S}asaki metrics}, Compos. Math. \textbf{147} (2011), no.~5,
  1613--1634. \MR{2834736}

\bibitem[Ler02]{Ler02a}
E.~Lerman, \emph{Contact toric manifolds}, J. Symplectic Geom. \textbf{1}
  (2002), no.~4, 785--828. \MR{2 039 164}

\bibitem[MSY08]{MaSpYau06}
Dario Martelli, James Sparks, and Shing-Tung Yau, \emph{Sasaki-{E}instein
  manifolds and volume minimisation}, Comm. Math. Phys. \textbf{280} (2008),
  no.~3, 611--673. \MR{MR2399609 (2009d:53054)}

\bibitem[NS12]{NiSe12}
Yasufumi Nitta and Ken'ichi Sekiya, \emph{Uniqueness of {S}asaki-{E}instein
  metrics}, Tohoku Math. J. (2) \textbf{64} (2012), no.~3, 453--468.
  \MR{2979292}

\bibitem[PP91]{PePo91}
Henrik Pedersen and Y.~Sun Poon, \emph{Hamiltonian constructions of
  {K}\"ahler-{E}instein metrics and {K}\"ahler metrics of constant scalar
  curvature}, Comm. Math. Phys. \textbf{136} (1991), no.~2, 309--326.
  \MR{1096118 (92f:32051)}

\bibitem[Sim92]{Sim92}
S.~R. Simanca, \emph{A note on extremal metrics of nonconstant scalar
  curvature}, Israel J. Math. \textbf{78} (1992), no.~1, 85--93. \MR{1194961
  (93k:58054)}

\bibitem[TF98]{To-Fr98}
Christina~Wiis T{\o}nnesen-Friedman, \emph{Extremal {K}\"ahler metrics on
  minimal ruled surfaces}, J. Reine Angew. Math. \textbf{502} (1998), 175--197.
  \MR{MR1647571 (99g:58026)}

\bibitem[WZ90]{WaZi90}
M.~Y. Wang and W.~Ziller, \emph{Einstein metrics on principal torus bundles},
  J. Differential Geom. \textbf{31} (1990), no.~1, 215--248. \MR{91f:53041}

\end{thebibliography}
\end{document}